\documentclass[11pt]{amsart}

\usepackage{tikz}
\usepackage{color}

\begin{document}
\title[Distinguishing index]{Asymmetric edge-coloring of graphs with simple automorphism group}
\author{Mariusz Grech, Andrzej Kisielewicz}
\address{Faculty of Pure and Applied Mathematics, Wrocław University of Science and Technology \\
Wybrzeże Wyspiańskiego Str. 27,
50-370 Wrocław, Poland}
\email{[mariusz.grech,andrzej.kisielewicz]@pwr.edu.pl}
\thanks{
{Supported in part by Polish NCN grant 2016/21/B/ST1/03079}}

\newtheorem{Theorem}{Theorem}[section]
\newtheorem{Lemma}[Theorem]{Lemma}
\newtheorem{Example}[Theorem]{Example}

\begin{abstract}
The distinguishing index $D'(\Gamma)$ of a graph $\Gamma$ is the least number $k$ such that $\Gamma$ has an edge-coloring with $k$ colors preserved only by the  trivial automorphism.  In this paper we prove that if the automorphism group of a finite graph $\Gamma$ is simple, then its distinguishing index $D'(\Gamma)=2$.    
\end{abstract}
 

\maketitle

{\color{black}

The distinguishing index  of a graph $\Gamma$ has been introduced by Pil\'sniak and Kalinowski \cite{KP} in 2015 to be the least number $d$ such that $\Gamma$ has an edge-coloring with $d$ colors breaking the symmetry of $\Gamma$ (i.e., such that no nontrivial automorphism of $\Gamma$ preserves this coloring). 
This is an analog to the notion of the distinguishing number $D(\Gamma)$ of a graph introduced by Albertson and Collins \cite{AC} in 1996, which has been defined in the same way for colorings of vertices.

Note that for asymmetric graphs we have $D(\Gamma)=D'(\Gamma)=1$. For other graphs  $D(\Gamma) \geq 2$, and it is conjectublack that almost all of them have the distinguishing number two (see \cite{CT,KP}). 
The  situation is similar for the distinguishing index, and the claim that  having the distinguishing index two is generic for asymmetric graphs has been supported by results in \cite{le,LPS,PT}.

The concepts of the distinguished number and distinguished index  generalize naturally to the distinguishing number of an arbitrary group action (\cite{tym, cha}).  Following this generalization,  it was realized in \cite{BC} that in permutation group theory the problem had been investigated for many years as a part of the study of set stabilisers of group actions, and some results  may be successfully applied. 
 In particular, a result by Gluck \cite{glu}  (obtained as early as in 1983) shows that if the order of the automorphism group of a graph $\Gamma$ is odd (and $> 1$), then both the distinguishing number and the distinguishing index of $\Gamma$ are two, $D(G)=D'(G)=2$. }

This paper is the sequel of \cite{GK}, where we have proven, in particular, that if the automorphism group of a finite graph $\Gamma$ is simple, then its distinguishing number $D(\Gamma) = 2$. To obtain this result we have described the distinguishing number for all possible actions of simple groups. Now, we apply the latter to show that for such graphs also the distinguishing index $D'(\Gamma) = 2$. This is not so straightforward 
{\color{black} 
as in the case of Gluck's result. }
It requires to consider some nonstandard action of the alternating groups 
and other simple groups, 
{\color{black}
and some special constructions of intransitive graphs. 
In fact, our paper illustrates a few phenomena that can occur when we approach the problem from the point of view of permutation groups.

At this point, we need to make a remark preventing quite frequent misunderstandings related to this topic.  
It is important to realize that even if two graphs have the same automorphism group their distinguishing numbers or indexes may be different. This is so, because these number depends on the way how the automorphism group acts (on vertices or edges),  not merely on its abstract algebraic structure. For example, both the clique $K_n$ and the star $K_{1,n}$ have the automorphism group isomorphic to the symmetric group $S_n$, but the reader can check that for $n\geq 6$, $D'(K_n)=2$ and $D'(K_{1,n})=n$. What counts here is that we have the same abstract group $S_n$, but two different actions of this group. 


 }

\section{Introduction}

{\color{black}
We assume that the reader is familiar with the basics of the group theory and permutation groups. The textbook \cite{DM} may serve as a standard reference.  
}

Let $G$ be a group acting on a set $X$ of points. 
The \emph{distinguishing number $D(G,X)$ of this action} is
the least number of colors needed to color the elements of $X$ such that the only color-preserving elements of $G$ are those that fix all points of $X$. 
{\color{black}
Note that if the action is faithful, then the only element of $G$ fixing all points in $X$ is the identity. Otherwise, all elements in the kernel of the action have this property.

In case, when $G$ is a permutation group and we consider its natural action on the base set $X$, then the action is faithful, and the distinguishing index $D(G,X)$ of the permutation group is the least number of colors needed to color the elements of $X$ so that no nontrivial permutation in $G$ preserves this coloring.

There are some concepts concerning the structure of permutation groups that are not standard and we need to explain them here.
}

A permutation group $G$ acting on a set $X$ is denoted $(G,X)$, or simply $G$, if the set $X$ is clear from the context. We write $xg$ for the image of $x\in X$ under the permutation $g\in G$.

Two permutation groups $(G,X)$ and $(H,Y)$ are 
\emph{permutation isomorphic} if there is a bijection $\lambda: X\to Y$ and an abstract group isomorphism $\psi : G\to H$ such that $\lambda(xg) = \lambda(x)\psi(x)$ for all $x\in X$ and $g\in G$. We write $G=H$ in such a case, and treat such groups as identical. We use known notations $S_n$ and $A_n$ for the symmetric and alternating groups on $n$ points, respectively. By $I_n$ we denote the trivial group on $n$ points, i.e., one consisting of the identity permutation only.  

For two permutation groups
$(G,X)$ and $(H,Y)$, by $G\oplus H$ we denote the \emph{direct sum} of these groups, i.e., the permutation group in which the product $G\times H$ acts on the disjoint union $X\cup Y$ by the formula
$$
x(g,h) =
\left\{\begin{array}{ll}
xg, & \mbox{if } x\in X\\
xh, & \mbox{if } x\in Y.
\end{array}\right.
$$
If we are given in addition an abstract isomorphism $\psi : G\to H$, then the \emph{parallel sum} $G||_\psi H$ of permutation groups $(G,X)$ and $(H,Y)$ is the permutation group consisting of permutations $(g,\psi(g))$ that acts on the disjoint union $X\cup Y$  by the formula
$$
x(g,\psi(g)) =
\left\{\begin{array}{ll}
xg, & \mbox{if } x\in X\\
x\psi(g), & \mbox{if } x\in Y.
\end{array}\right.
$$ 
Let us note that $G||_\psi H$ is a subgroup of $G\oplus H$. Moreover, 
both the operations $G\oplus H$ and $G||_\psi H$ are commutative, and associative (due to the fact that we consider permutation groups up to permutation isomorphism). If $G=H$ and the isomorphism $\psi$ the identity, we write $G^{(2)}$ for $G||_\psi G$, and more generally,  $G^{(k)}$ in the case of $k\geq 1$ summands. We call it the \emph{parallel multiple} of the permutation group $G$ and adopt the convention  $G^{(1)} = G$. (We refer the reader to \cite{GK} for a more general construction).

{\color{black} We note that if $k\neq m$, then $G^{(k)}$ and $G^{(m)}$ may have different distinguishing numbers (see Theorem~\ref{th:GK} below).}
Note also that $I_n=I_{n-r}\oplus I_r$, for any $0< r < n$. If $G$ has $k$ fixed points, then $G=G_0\oplus I_k$, where $G_0$ has no fixed points. The following example will be used later in our proof.

\begin{Example}\label{ex:1}
{\rm 
Let $\Gamma$ be a graph whose connected components $\Gamma_1,\dots, \Gamma_r$ ($r>1$) are pairwise isomorphic. Assume also that these components are asymmetric, that is, $Aut(\Gamma_i) = I_m$, where $m$ is the order of the component. {\color{black} Since there are no asymmetric graphs of order $2$ through $5$,} 
this is possible only when $m>5$ or $m=1$.
Then, all automorphisms of $\Gamma$ permute the whole components. More precisely, the vertices of $\Gamma$ may be partitioned into subsets  $X_j =\{x_j^1, x_j^2,\ldots, x_j^r\}$, $x_j^i\in \Gamma_i$, $j\leq m, i\leq r$, such that each permutation of the components permutes elements in each $X_j$ in a parallel way. Thus, $Aut(\Gamma)=S_r^{(m)}$. 

For $m=1$ this is $S_r^{(1)}= S_r$. For $m>5$ and $r>2$ this group has a nontrivial normal subgroup $A_r^{(m)}$. For $r=2$, $Aut(\Gamma)=S_2^{(m)}$ is a simple group abstractly isomorphic to the two element group $Z_2$. 
It is easy to see that in the latter case, for $m>1$, $D(\Gamma)=D'(\Gamma)=2$. 
{\color{black}
Indeed, to obtain a requiblack distinguishing coloring it is enough to color black one vertex (resp. edge) in one of the two components.
}
}\end{Example}

{\color{black}
Recall that a group $G$ is simple if it has no nontrivial normal subgroup. There is a huge literature on finite simple groups connected with the classification of finite simple groups and some of these results have been used in our previous paper \cite{GK}. In this paper we apply only the main result of \cite{GK}. It needs some preliminary explanation.

There are many various actions of simple groups, transitive and intransitive, and our result \cite{GK} describes all actions with the distinguished number larger than 2. Since the kernel of the action of a group $G$ is a normal subgroup, it follows that all actions of simple groups are faithful. Therefore we may treat them as permutation groups and consider permutation groups rather than actions. 

By Frucht's theorem we know that for each group $G$ there exists a graph $\Gamma$ whose automorphism group $Aut(\Gamma)$ is \emph{isomorphic} to $G$. Yet, given a permutation group $(G,X)$ we do not know, in general, if there exists a graph $\Gamma=(X,E)$ (on the set of vertices $X$) such that $Aut(\Gamma)=(G,X)$. This problem, known as the concrete version of K\"onig's problem, remains largely unsolved. 
On the other hand, we know that for each group $G$ there exists a regular permutation group abstractly isomorphic to $G$ (given by a regular action of $G$ on itself). By virtue of \cite{god}, each regular permutation group (with some listed exceptions) admits a graphical regular representation (GRR), i.e., it is the automorphism group of some graph. In particular, this holds for all simple groups.
Using this and the constructions of the parallel sum one can obtain easily further examples of intransitive graphs with simple automorphism groups.  Also, various more involved constructions lead to such graphs.  Yet, in general, we know very little about graphs whose automorphism groups are simple (see examples and an open problem on this topic in \cite{GK}). Our approach in this paper is to consider only those simple permutation groups that have the distinguishing number larger than~$2$ and examine them whether they may or may not be the automorphism groups of graphs. For this we need the result formulated below.

 }


In the formulation, we use the standard notation for abstract groups to denote permutation groups in their natural action (see e.g. \cite{DM}). If an action is other than natural, then in the notation
we add the number of points the group in question acts on (in all the cases pointed out this is enough to identify the action). For example, $(L_2(11),11)$ denotes the unique permutation group where the special projective linear group $L_2(11)$ acts on $11$ points (rather than on $12$, as in the standard action). The notation $A_6||_\psi A_6$ is used for the unique parallel sum, where $\psi$ is the exceptional nonpermutation automorphism of $A_6$ (for details see \cite{GK}). We write $D(G)$ for $D(G,X)$.
Recall that fixed points do not affect the distinguishing number: if $G=G_0\oplus I_k$, then  $D(G)=D(G_0)$.

\begin{Theorem}\cite[Theorem 4.1]{GK} \label{th:GK}
Let $G$ be a simple permutation group with no fixed points. Then,
$D(G) = 2$ except for the following cases: 
\begin{enumerate}
 \item 
If $G = A_n^{(k)}$ $(n\geq 5)$, then $D(G)$ is the smallest integer $d$ such that $d^k \geq n - 1$;
\item If $G \in \{L_3(2), M_{11}, M_{12} \}$, then $D(G) = 4$.
\item If $G \in \{ L_2(5), L_2(7), L_2(8)$, $(A_6 , 10), (L_2(11), 11)$, $(M_{11}, 12)$, $L_3(3)$, $(A_8, 15)$, $M_{22}$, $M_{23}, M_{24}, A_6||_\psi A_6\}$ then $D(G) = 3$.
\end{enumerate}
\end{Theorem}

\section{Distinguishing index}

{\color{black}
Observe that for a graph $\Gamma=(V,E)$ we have $D(\Gamma) = D(Aut(\Gamma),V)$ and $D'(\Gamma) = D(Aut(\Gamma),E)$. In fact, for the distinguishing index the definition \emph{via} the action of the automorphism group is a little bit more general. It includes also the graphs having automorphisms of order $2$ that leave all the edges fixed, while in the original formulation of \cite{KP}, the value $D'(\Gamma)$ is undefined in such cases.

}

In particular, the difference concerns the trivial case, when $\Gamma=K_2$ consists of one edge. To avoid considering this special trivial case in our proofs, in the rest of the paper we assume that $\Gamma$ is of size $>1$, that is it has at least two edges. 
Our result is the following.

\begin{Theorem}\label{th:main} 
Let $\Gamma$ be a nontrivial graph of size $>1$. If the automorphism group  $Aut(\Gamma)$ of $\Gamma$ is simple, then the distinguishing index \hbox{$D'(\Gamma) = 2$.} 
\end{Theorem}

In order to prove this result we need some lemmas. In all of them we assume that $\Gamma=(X,E)$ has at least two edges. 
First we blackuce the problem to the connected graphs.
{\color{black} (Note that in contrast with the distinguishing number, such a blackuction is not trivial. It is true that a graph $\Gamma$ and its complement $\overline{\Gamma}$ have the same automorphism group, and therefore $D(\Gamma) = D(\overline{\Gamma})$. Yet, the distinguishing indexes may be different. For example for a star of order $5$ we have $D'(K_{1,4})=3$ and $D'(\overline{K_{1,4}})=4$.)

}

\begin{Lemma} 
If $\Gamma$ is not connected and its automorphism group is simple, then one of the following holds
\begin{enumerate}
\item there exists a connected component $\Gamma_0$ of $\Gamma$, such that $Aut(\Gamma) = Aut(\Gamma_0)\oplus I_k$ for some $k>0$, and $D'(\Gamma)\leq D'(\Gamma_0)$.
\item $|Aut(\Gamma)|=2$  and $D'(\Gamma)=2$. 
\end{enumerate}
\end{Lemma}

\begin{proof}
 
Let $\Gamma_1,\ldots,\Gamma_r$  be the connected components of $\Gamma$, and $X_1,\dots,X_r$ the corresponding partition of $X$ ($r>1$). Let $G=Aut(\Gamma)$ be the automorphism group of $\Gamma$, and $H$ its subgroup  consisting of the automorphisms preserving the connected components of $\Gamma$. It is obvious that $H= Aut(\Gamma_1) \oplus \ldots \oplus Aut(\Gamma_r)$. Moreover, since each automorphism in $G$ preserves the partition into connected components, $H$ is a normal subgroup of $G$. As $G$ is simple, it follows that either $H=G$ or $H$ is the trivial subgroup. 

If $H=G$, then $G$ is a direct sum and each summand $Aut(\Gamma_i)$ is a normal subgroup of $G$. It follows that all these summands but one are trivial (i.e., equal to $I_s$ for some $s$), and exactly one is nontrivial. Consequently, $G=Aut(\Gamma_i)\oplus I_k$ for some $i$ and $k>0$. Moreover, in such a case, each edge-coloring of $\Gamma_i$ breaking the symmetry of $\Gamma_i$ breaks the symmetry of $\Gamma$, and therefore $D'(\Gamma)\leq D'(\Gamma_i)$, as requiblack.

We consider the second case, when $H$ is trivial,  i.e., $H=I_n$. Then each of $Aut(\Gamma_i)$ is trivial, and the only nontrivial automorphisms of $\Gamma$ are those permuting the whole components. In such a case it is easy to describe the structure of $G=Aut(\Gamma)$. 

First, let us assume that all the connected components of $\Gamma$ are isomorphic graphs of order $m\geq 1$. Then the structure is described in Example~\ref{ex:1}.  Thus, the only case when $G$ is simple, is $G=S_2^{(m)}$ and $m=1$ or $m>5$. 
Then, if $m>5$,  the order $|G|=2$,  and $D'(\Gamma)=2$, as requiblack. If $m=1$ and $G=S_2$, then $\Gamma$ is of size $1$, which is excluded from our consideration.

Finally, assume that $H$ is trivial and there are nonisomorphic connected components of $\Gamma$. Then, considering the partition into subsets of isomorphic components, we see that on each such subset the automorphism group acts as described above, independently of other subsets. In consequence,
$$G=S_{r_1}^{(m_1)} \oplus\ldots\oplus S_{r_s}^{(m_s)}$$ with $s>1, r_i>0, m_i>0$. Note that it may happen that $r_i=1$, in which case $S_{r_i}^{(m_i)}=I_{m_i}$.  Without loss of generality we may assume that $r_i > 1$ for all $i<s$, and $r_s\geq 1$ (we made use of the fact that $I_{m_i}\oplus I_{m_j} = I_m$, for some $m$).
Such a group is usually not simple. It has a nontrivial normal subgroup of the form $S_{r_1}^{(m_1)} \oplus I_{k}$, where $k=n-r_1m_1$. The only case when such a group is simple is when it is of this form  itself, i.e., 
$G =S_{2}^{(m)} \oplus I_k$. Then, as before, $|G|=2$ and $D'(G)=2$, completing the proof.
\end{proof}

By the virtue of this lemma, as we wish to prove that $D'(\Gamma)=2$, we may restrict our further study to connected graphs. For the further proof we will need the following observation. There are only two kinds of permutation groups among those listed in Theorem~\ref{th:GK} with $D(G)>2$ that are intransitive. These are  $A_n^{(k)}$, $k>1,n\geq 5$ or $A_6||_\psi A_6$. In both cases they have orbits of equal sizes. The remaining groups are transitive. So we can speak of the (uniquely defined) \emph{size of the orbit} for each of the groups on that list; for transitive groups this size is equal to the degree. We have

\begin{Lemma}\label{l:}
Let $G$ be one of the simple permutation groups listed in Theorem~\ref{th:GK}. Let $n$ denotes the size of the orbit in $G$. Then the only proper divisor $d$ of $2n$, $1<d<2n$, such that $G$ has a transitive action on a set of cardinality $d$ is $d=n$. 
\end{Lemma}
\begin{proof}
First observe that the claim is true for intransitive groups on the list, $A_n^{(k)}$ and $A_6||_\psi A_6$. This is so, since there is no nontrivial action of $A_n$ on a set of cardinality $d < n$, which implies easily that there is no nontrivial action of $A_n$ on a set of cardinality $d < 2n$, where $d$ is a divisor of $2n$ other than $n$.  

If $(G,X)$ is transitive and has a transitive action on a set of cardinality $d$, then this action is equivalent to the action of $G$ on the $H$-cosets of $G$, where $H=G_x$ is the stabilizer of a point $x\in X$. Then, the index $|G :H| = d$ (cf. \cite[p.~22]{DM}).  This is equivalent to the existence a subgroup of $G$ of cardinality $|G|/d$. 

One checks directly (using GAP or other system of combinatorial computation) that for none of the transitive permutation groups listed in Theorem~\ref{th:GK} exists a subgroup of such cardinality. 
For example, consider $G =(A_6,10)$. Then $|A_6|=360$ and the proper divisors of $2n=20$ are $2,4$ and $5$. Yet, we know that  $A_6$ has no subgroup of cardinality $180, 90$ or $72$.
\end{proof}

Note that the assumption that $d$ is a divisor is essential. For example,  in the action of $A_6$ on $n=10$ points we know that there is an action on $6$ points corresponding to 
a subgroup of $A_6$ of cardinality $60$. Yet, $6$ is not a divisor of $20$.

For the next lemma we establish additional terminology. Let $\Gamma = (X,E)$ be a graph, and $Aut(\Gamma)$ its automorphism group. Then by the \emph{orbits} of $\Gamma$ we mean the orbits of $Aut(\Gamma)$, and by the \emph{edge-orbits} of $\Gamma$ we mean the orbits of $Aut(\Gamma)$ in its action on the edges of $\Gamma$. The edge-orbits should be distinguished from the orbitals of  $Aut(\Gamma)$, which are the orbits of $Aut(\Gamma)$ in its action on two-element subsets.  Each edge-orbit is an orbital, but an orbital needs not to contain any edge. The fixed points of $Aut(\Gamma)$ form orbits of cardinality $1$, which are called \emph{trivial}.

We will need one more technical lemma concerning a special kind of graphs. 
Given $n>2$, call a graph $\Gamma$ \emph{$n$-uniform} if each nontrivial orbit of $\Gamma$ has size $n$, and each edge-orbit has the same size $n$.  We have the following.

\begin{Lemma}\label{l:uni}
Let $\Gamma$ be an $n$-uniform graph, $n>3$, and $G=Aut(\Gamma)$ its automorphism group. Assume that 
$\Gamma$ has $r\geq 1$ nontrivial orbits  and  $m\geq 0$ is the number of  fixed points of $G$. If the action of  $G$ on each nontrivial orbit is $2$-transitive, then $G= S_n^{(r_1)} \oplus\ldots\oplus S_n^{(r_s)}\oplus I_m$ for some $s\geq 1$ and $r_1+\ldots+r_s=r$ $($for $m=0$ the component $I_m$ is absent$)$.
\end{Lemma}

\begin{proof}
First note that if $Y$ is a nontrivial orbit of $\Gamma$, then since $G$ on $Y$ is $2$-transitive, the graph $\Gamma$ restricted to $Y$ is either empty or complete. The latter however is excluded, since $|Y|=n>3$ and each edge-orbit has size $n$.

Now, if $\Gamma$ has only one nontrivial orbit $Y$, and $x$ is a vertex fixed by $G=Aut(\Gamma)$ such that there is an edge joining $x$ an a vertex $y\in Y$, then since $G$ is $2$-transitive on $Y$, each vertex in $Y$ is joined by an edge with  $x$. It follows that $G = S_n\oplus I_m$, as requiblack. 

So, assume that $r>1$. Then the situation with fixed points is similar. If $Y$ is one of nontrivial orbits and $x$ is a fixed point, then either all vertices in $Y$ are joined by an edge with $x$ or none of them. 

Let $Z$ be another nontrivial orbit and let $\{y,z\}$ be an edge joining $y\in Y$ and $z\in Z$.
Since $G$ is transitive on $Y$ and $Z$, it follows that the edge-orbit $O$ containing $\{y,z\}$ consists of $n$ independent edges. 

Call the edges that join two nontrivial orbits \emph{essential}, to distinguish them from the edges incident to fixed points. We wish to show that for any two different nontrivial orbits $Y$ and $Z$ that are connected by a path of essential edges, the whole system of essential edges determines uniquely a one-to-one correspondence between $Y$ and $Z$. 

Let $y\in Y$ and $z\in Z$ be such that there is a path consisting of essential edges joining $y$ and $z$. Suppose also that there is another vertex $x\in Z$, $x\neq z$, such that there is a path consisting of essential edges joining $y$ and $x$. Since $G$ is 2-transitive on $Z$, there exists $g\in G$ such that $zg=z$  and $xg \neq x$. 

Now, by what we have assumed, there exists a path from $z$ to $x$ consisting of essential edges. Let $(z,z_1)$ be the first edge on this path. Consider the image $(z,z_1)g = (zg,z_1g)$ of the edge $(z,z_1)$ under $g$. Since $zg=z$ and the edge-orbit containing $(z,z_1)$ consists of independent edges, it follows that $z_1g=z_1$. This argument works for every next edge on the path from $z$ to $x$. Consequently, $xg=g$, which contradicts the assumption on $g$.

It follows that if  $Y$ and $Z$ are two nontrivial orbits that are joined by a path consisting of essential edges, then for each vertex $y\in Y$ there exist a unique vertex $z\in Z$ such that there is a a path of essential edges joining $y$ and $z$. Moreover this forms a one-to-one correspondence between vertices of $Y$ and $Z$. In particular, if there is and edge joining $Y$ and $Z$, then there are exactly $n$ edges joining $Y$ and $Z$; they are independent and  form an edge-orbit of $\Gamma$.

Now we can observe that for every permutation of vertices in a nontrivial orbit $Y$ there exist corresponding permutations of vertices in other nontrivial orbits such that their composition preserves essential edges, and in consequence all edges of $\Gamma$. In other words, every permutation of $Y$ may be extended to an automorphism of $\Gamma$. 

More precisely, let $\Gamma^*$ be the graph obtained from $\Gamma$ by deleting the fixed points of $\Gamma$. Then the edges of $\Gamma^*$ are exactly the essential edges of $\Gamma$. Suppose that $\Gamma^*$ has $s\geq 1$ connected components, and let $\Gamma_i$ be a connected component of $\Gamma^*$ containing $r_i\geq 1$ nontrivial orbits of $\Gamma$. Then, since all orbits of $\Gamma_i$ are connected by essential edges, $Aut(\Gamma_i) = S_n^{(r_i)}.$ Observe that automorphisms of various connected components $\Gamma_i$ can be composed to form an automorphism of $\Gamma$, and each automorphism of $\Gamma$ is of this form. Consequently, $G= S_n^{(r_1)} \oplus\ldots\oplus S_n^{(r_s)}\oplus I_m$, as requiblack.

\begin{figure}\label{fig1}
\begin{tikzpicture}[auto,inner sep=1pt, 
minimum size=2pt]

  
  \draw  (1,1) ellipse (6mm and 7mm);
    \draw  (4,1) ellipse (6mm and 7mm);
   \draw  (8,1) ellipse (6mm and 7mm); 
      \draw  (2.5,3) ellipse (6mm and 7mm); 
      
      
     %
     \draw (4.5,3.3)  node[circle,fill=black,draw]{}; 
     \draw[-] (6,1)--(4.5,3.3); 
      \draw (6,1) node[circle,fill=black,draw]{};

\draw[-] (1.1,1.75)--(1.9,2.8);       
\draw[-] (1.26,1.69)--(1.96,2.59); 
\draw[-] (1.4,1.6)--(2.06,2.44);
\draw[-] (1.5,1.45)--(2.2,2.35);

\draw[-] (3.9,1.75)--(3.1,2.8);
\draw[-] (3.74,1.69)--(3.04,2.59);
\draw[-] (3.6,1.6)--(2.94,2.44);
\draw[-] (3.5,1.45)--(2.8,2.35);
      
\draw[-] (1.63,1.28)--(3.37,1.28);
\draw[-] (1.66,1.1)--(3.34,1.1); 
\draw[-] (1.66,0.94)--(3.34,0.94); \draw[-] (1.63,0.76)--(3.37,0.76);

 \draw[-] (6,1)--(4.6,1.28); 
  \draw[-] (6,1)--(4.65,1.1); 
   \draw[-] (6,1)--(4.65,0.94); 
    \draw[-] (6,1)--(4.6,0.76); 
    
    \draw[-] (6,1)--(7.4,1.28); 
 \draw[-] (6,1)--(7.35,1.1); 
   \draw[-] (6,1)--(7.35,0.94); 
    \draw[-] (6,1)--(7.4,0.76); 
  

  \draw[-] (4.5,3.3)--(2.9,3.6); 
   \draw[-] (4.5,3.3)--(3.03,3.42); 
    \draw[-] (4.5,3.3)--(3.1,3.25); 
       \draw[-] (4.5,3.3)--(3.15,3.05);


 
\draw (1,1) node {$X_1$}; 
\draw (2.5,3) node {$X_2$};  
\draw (4,1) node {$X_3$};  
\draw (8,1) node {$X_4$};  
\draw (6,0.7) node {$x$};  
\draw (4.6,3.6) node {$y$}; 
\draw (6,4) node {}; 


  \end{tikzpicture}
\caption{An $n$-uniform graph $\Gamma$ with $Aut(\Gamma) = S_n^{(3)} \oplus I_2 \oplus S_n$.}
\end{figure}
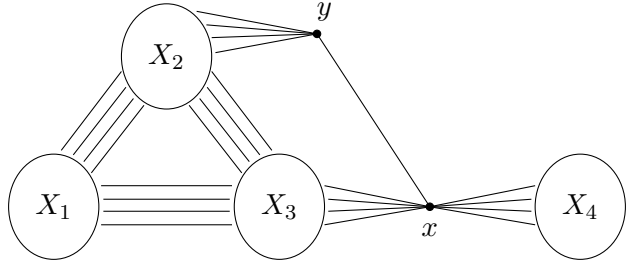

An illustration is given in Figure~1. It contains a sketch of an $n$-uniform graph with $r=4$ nontrivial orbits $X_i$  (each of the size $n$), $m=2$ fixed points $x$ and $y$, and $s=2$ nontrivial connected components in $\Gamma^*$. One can easily see that $Aut(\Gamma) =  S_n^{(3)}\oplus I_2\oplus S_n$.  
\end{proof}

Now we are ready to prove the following lemma dealing with connected graphs that are intransitive.

\begin{Lemma}
If  $\Gamma$ is connected and intransitive and $Aut(\Gamma)$ is simple, then $D'(\Gamma)=2$. 
\end{Lemma}
\begin{proof}
We consider the action of $G=Aut(\Gamma)$ on the set $E$ of the edges of $\Gamma$. Since all nontrivial actions of simple groups are faithful, we have a permutation group $(G,E)$. We prove that if $D(G,E) > 2$, then $(G,X)$ is not the automorphism group of any graph.

Let us assume, to the contrary,  that $(G,E)$ is one of the permutation groups on the list of Theorem~\ref{th:GK}, possibly with some fixed points added. Then, by the remark preceding Lemma~\ref{l:},  all  orbits of $(G,E)$  have the same size $n\geq 5$.


Consider the action of $G$ on $X$. (Note that while $D(G,E)>2$, the action of $G$ on $X$ may be different and may have a different distinguishing number). Since, by assumption, $(G,X)$ is intransitive, it has at least two orbits, and at least one of them is nontrivial. Let $Y$ be a nontrivial orbit of $(G,X)$. Then, since $\Gamma$ is connected and intransitive, there exists another orbit $Z$ (possibly trivial) such that there is an edge in $\Gamma$ joining $Y$ and $Z$.  Let $O$ denotes the edge-orbit containing this edge. Note that, if $d>1$ is the number of the edges in $O$ joining some point $y\in Y$ with $Z$, then each point in $Y$ is joined with $Z$ by exactly $d$ edges from $O$. Consequently, the cardinality $|O|=d|Y|$. 
Since  $|O|=n$ is the size of the orbit in $(G,E)$, we infer that $|Y|$ is a divisor of $n$, and by Lemma~\ref{l:}, $|Y|=n$. It follows that all the orbits in $(G,X)$ are of the size $n$ or $1$. 

Thus we have proved  that the graph $\Gamma$ is $n$-uniform. By Lemma~\ref{l:uni}, the automorphism group of $\Gamma$ is not simple. Indeed, the direct product of nontrivial components is not simple, and if $G=S_n^{(r)}+I_m$, then it has a normal subgroup $A_n^{(r)}+I_m$. This is a contradiction proving the lemma.
\end{proof}

To complete the proof of Theorem~\ref{th:main} it remains to prove the following.

\begin{Lemma}
If a connected graph $\Gamma$ is transitive and $Aut(\Gamma)$ is simple, then $D'(\Gamma)=2$. 
\end{Lemma}
\begin{proof}
As in the previous proof, for $G=Aut(\Gamma)$, we assume that $(G,E)$ is one of the permutation groups on the list of Theorem~\ref{th:GK}, possibly with some fixed points added, and $n\geq 5$ is the size of the orbit in $(G,E)$. 
We show that $(G,X)$ is not the automorphism group of any graph.

Let $O$ be an edge-orbit of $(G,X)$, and let $m$ be the number of edges adjacent to a vertex $x\in \Gamma$ belonging to $O$. Since 
$(G,X)$ is now transitive, this number is the same for every vertex $x\in \Gamma$. Therefore, the cardinality $n=m|X|/2$.  
By Lemma~\ref{l:}, either $m=1$ and $|X|=2n$ or $m=2$ and $|X|=n$.  

The latter leads immediately to a contradiction. Indeed, since up to equivalence there is only one action of $G$ on $n$ elements,  $(G,X)$ is $2$-transitive, and the only $2$-transitive automorphism group of a graph is $G=S_n$, which is not simple. 

Consider the case when $|X|=2n$. 
If $(G,E)$ is transitive, then the size of the orbit $n=|E|$, and since $\Gamma$ is transitive, $E$ consists of $n$ independent edges. However, this contradicts the fact that $\Gamma$ is connected. So we may assume that $(G,E)$ is not transitive. Then by Theorem~\ref{th:GK}, either 
$(G,E)=A_n^{(k)}$ or $A_6 ||_\psi A_6$, that is, $G$ is abstractly isomorphic  to $A_n$. So, since $\Gamma$ is transitive, all we need is to check all transitive actions of $A_n$, $n\geq 5$, on $2n$ points. 

For $A_5$ we check directly that there exists a transitive action of $A_5$ on $10$ elements, but this action has two orbitals of sizes $15$ and $30$,  so it cannot have edge-orbits of size $n$. 
For $A_6$ there is no action on $12$ points, at all. 
For $n>6$ we apply \cite[Theorem~5.2A]{DM}  describing all subgroups of $A_n$ of small index (cf. \cite{lieb}). Again we use the correspondence between the transitive actions of $G$ and its subgroups. Using this, we show that for $n>6$ there is no subgroup of $A_n$ of index $2n$. In fact, we
need only the following conclusion from \cite[Theorem~5.2A]{DM} concerning the index alone. If $A=A_n$ is an alternating group with $n\geq 5$, and $G$ is its subgroup of index $|A_n:G| < \binom{n}{r}$ for some $1\leq r\leq n/2$, then one of the following holds:
\begin{enumerate}
    \item[(i)] $\binom{n}{s} \leq |A_n:G| \leq \binom{n}{s}s!$ for some $s<r$,
    \item[(ii)] $n=2m$ is even and $|A_n:G|= \frac{1}{2}\binom{n}{m}$, or
    \item[(iii)] (\emph{exceptional cases}) the pair $(n,|A_n:G|)$ belongs to the following set  $\{(6,15),(5,6), (6,6), (7,15), (8,15), (9,120)\}$.
\end{enumerate}
We check for possibilities that $|A_n:G| = 2n$ for $n>6$.
For (i), if $s>1$, then  we have
$|A_n:G| \geq \binom{n}{s}\geq \binom{n}{2} > 2n$, and if $s=1$ we get $|A_n:G|=n$.
Hence for $n>6$ there are no subgroups $G$ of $A_n$ of index $2n$ satisfying (i). 
For (ii), 
if $n=2m>6$, then  $\frac{1}{2}\binom{n}{m} > 2n$, which means that also no subgroup satisfying this condition has index $2n$. Finally, we see that no pair in (iii) is of the form $(n,2n)$, which completes the proof. 
\end{proof}

\end{document}